\documentclass[12pt, letterpaper]{article}
\usepackage[utf8]{inputenc}
\usepackage[english]{babel}
\usepackage{amsmath}
\usepackage{amsfonts}
\usepackage{amssymb}
\usepackage{amsthm}
\usepackage{graphicx}
\usepackage[left=1in,top=1in,right=1in,bottom=1in]{geometry}

\usepackage[usenames]{color}

\newtheorem{teo}{Theorem}
\newtheorem{lem}{Lemma}

\newtheorem{cor}{Corollary}
\newtheorem{remark}{Remark}
\newtheorem{con}{Condition}

\numberwithin{equation}{section}
\title{Convergence Rate for The Number of Crossing in a Random Labelled Tree}
\author{Santiago Arenas-Velilla \\\small Centro de Investigaci\'on en Matem\'aticas,\\[-0.8ex]
\small Guanajuato, Gto. 36000, Mexico\\
\small\tt santiago.arenas@cimat.mx \\\and Octavio Arizmendi\\ \small Centro de Investigaci\'on en Matem\'aticas,\\[-0.8ex]
\small  Guanajuato, Gto. 36000, Mexico\\
\small\tt octavius@cimat.mx\\}
\begin{document}
\maketitle

\begin{abstract}
    We consider the number of crossings in a random labelled tree with vertices in convex position. We give a new proof of the fact that this quantity is  asymptotically Gaussian with mean $n^2/6$ and variance $n^3/45$. Furthermore, we give an estimate for the Kolmogorov distance to a Gaussian distribution which implies a convergence rate of order $n^{-1/2}$.
\end{abstract}

\textbf{keywords}: crossings, random labelled trees, normal approximation.

\section{Introduction}

Random trees have been broadly studied in many directions.  Here we are interested in the number of crossings in a random labelled tree in convex position.
As proved by Arizmendi et al. \cite{arizmendi2019number}, this quantity satisfies a normal approximation as the number of vertices goes to infinity. This note considers a quantitative version of such result.

\begin{teo}
Let $X_n$ be the number of crossing in a random labelled tree with $n$ points in convex position. Then, as $n$ goes to infinity, $X_n$ approaches a normal distribution with mean $\approx n^2/6$ and variance $\approx n^3/45$ with convergence rate $n^{-1/2}$. Moreover, we have 
\begin{equation*}
\sup_{z \in \mathbb{R}} |\mathbb{P}\left(\frac{X_n-n^2/6}{\sqrt{n^3/45}} \leq z\right) - \mathbb{P}(Z \leq z)| \leq \frac{C}{\sqrt{n}},
\end{equation*}
where $Z$ is a standard Gaussian random variable and $C$ is a constant independent of $n$.
\end{teo}

We will use Stein's method, in the form of the size biased transform developed by Goldstein y Reinert \cite{goldstein1997stein}. We shall mention that our methods are similar to Paguyo's \cite{paguyo2021convergence}  where he gives a quantitative version of the classical result of Flajolet and Noy \cite{FlajoletNoy2000}.

\section{Preliminaries}

\subsection{Size Bias Transform}

Let $X$ be a positive random variable with mean $\mu$ finite. We say that the random variable $X^s$ has the \textit{size bias} distribution with respect to $X$ if for all $f$ such that $\mathbb{E}[Xf(X)] < \infty$, we have 
$$\mathbb{E}[Xf(X)] = \mu \mathbb{E}[f(X^s)].$$
In the case of $X = \sum_{i=1}^{n} X_i$, with $X_i$'s positive random variables  with finite mean $\mu_i$, there is a recipe to construct $X^s$ (Proposition 3.21 from \cite{ross2011fundamentals}) from the individual size bias distributions of the summands $X_i$:

\begin{enumerate}
\item For each $i = 1, \ldots, n$, let $X_i^{s}$  having the size bias distribution with respect to $X_i$, independent of the vector $(X_j)_{j \neq i}$ and $(X_j^s)_{j\neq i}$. Given $X_i^s = x$, define the vector $(X_{j}^{(i)})_{j \neq i}$ to have the distribution of $(X_j)_{j\leq i}$ conditional to $X_i = x$.
\item Choose a random index $I$ with $\mathbb{P}(I = i) = \mu_i / \mu$, where  $\mu = \sum \mu_i$, independent of all else. 
\item Define $X^s = \sum_{j \neq I} X_{j}^{(I)} + X_I^s$.
\end{enumerate}
It is important to notice that the random variables are not necessarily independent or have the same distribution. Also, $X$ can be an infinite sum (See Proposition 2.2 from \cite{chen2011normal}). 

If $X$ is a Bernoulli random variable, we have that $X^s = 1$. Indeed, if $\mathbb{P}(X=1)=p$, $\mathbb{E}(X)=p=\mu$ and then $$\mathbb{E}[Xf(X)] =(1-p)(0f(0))+p(1f(1))=pf(1)=\mu f(1)=\mu \mathbb{E}[f(1)].$$
Therefore, we have the following corollary (Corollary 3.24 from \cite{ross2011fundamentals}) by specializing the above recipe. 
\begin{cor}
\label{corollary_sizebiasBernoulli}
Let $X_1, X_2, \ldots, X_n$ be Bernoulli random variables with parameter  $p_i$. For each $i = 1, \ldots, n$ let $(X_j^{(i)})_{j \neq i}$ having the distribution of $(X_j)_{j \neq i}$ conditional on $X_i =1$. If $X = \sum_{i=1}^n X_i$, $\mu = \mathbb{E}[X]$, and $I$ is chosen independent of all else with $\mathbb{P}(I= i) = p_i /\mu$, then $X^s = 1+\sum_{j \neq I} X_j^{(I)}$ has the size bias distribution of $X$. 
\end{cor}

The following result (Theorem 5.3 from \cite{chen2011normal}) gives us bounds for the Kolmogorov distance, in the case that a bounded size bias coupling exists. This distance is given by
$$d_{Kol}(X,Y) := \sup_{z \in \mathbb{R}} |F_X(z) - F_Y(z)|,$$
where $F_X$ and $F_Y$ are the distribution functions of the random variables $X$ and $Y$.
\begin{teo}\label{teoKolmo}
Let $X$ be a non negative random variable with finite mean $\mu$ and finite, positive variance $\sigma^2$, and suppose $X^s$, have the size bias distribution of $X$, may be coupled to $X$ so that $|X^s-X| \leq A$, for some $A$. Then with $W = (X- \mu)/ \sigma$,

\begin{equation}\label{Kolmogorov}
d_{Kol}(W,Z) \leq   \frac{6 \mu A^2}{\sigma^3} +\frac{2 \mu \Psi}{\sigma^2},
\end{equation}
where $Z$ is a standard Gaussian random variable, and $\Psi$ is given by 
\begin{equation}\label{Psi}
\Psi = \sqrt{\mathrm{Var}(\mathbb{E}[X^s-X | X])}
\end{equation}

\end{teo}

\subsection{Random Trees}

The main object of study in this paper is the uniform random labelled tree. We will be interested in the number of crossing when the vertices are in convex position. In order to get that information, we need the probabilities of certain configurations having such crossing. This  translates in the probability of having certain forests as subgraphs.

Let $T_n$ be a uniform random labelled tree on $n$ vertices, that is a tree chosen uniformly at random over the set of $n^{n-2}$ possibles trees on $n$ vertices with its vertices labelled. Using Pitman's techniques (see \cite{MR537284}) one can see that the probability that a random labelled tree contains certain forest $t_1, t_2, \ldots, t_n$ is given by\begin{equation} \label{eq:probaofsubtrees}\mathbb{P}(t_1,t_2, \ldots, t_n) = \frac{1}{n^{\mathcal{E}}} \prod_{i=1}^n v_i,\end{equation}
where $\mathcal{E}$ is the number of edges of the forest $t_1, t_2, \ldots, t_n$, and $v_i$ is the number of vertices of each subtree $t_i$. In particular, this shows that containing subtrees which are disjoint are independent events.

\section{Proofs}

Let $T_n$ be a uniform random labelled tree with its vertices in convex position on the set $[n]= \{1, 2, \ldots, n\}$, and we denote by $X_n := X_n(T_n)$ the random variable that counts the number of crossing of $T_n$. In this case, we write $X_n$ as a sum of $\binom{n}{4}$ Bernoulli variables corresponding to the different possible crossings. That is,
\begin{equation}
\label{crossing_number}
X_n = \sum_{1 \leq a <b < c< d\leq n} \mathbb{I}_{a \sim c} \mathbb{I}_{b \sim d} = \sum_{j \in \binom{n}{4}} \mathbb{I}_{ \{T_n \text{ has a crossing in } j\} }= \sum_{j \in \binom{n}{4}} Y_j.
\end{equation}

\subsection{Mean and Variance of $X_n$}

Directly from \eqref{eq:probaofsubtrees}, we notice $Y_j$ is a Bernoulli random variable having probability of success
$$\mathbb{P}(Y_j = 1) = \frac{4}{n^2},$$
from where we obtain that 
\begin{equation}\label{valoresperadocruces}
\mathbb{E}(X_n) = \binom{n}{4} \frac{4}{n^2}= \frac{(n-1)(n-2)(n-3)}{6n} \sim \frac{n^2}{6}.
\end{equation}
One can further calculate the variance of the number of crossings $X_n$, by expanding $X_n^2$ as a double sum of crossings and splitting by cases to obtain (see  \cite{arizmendi2019number}),
\begin{equation}\label{varianzacruces}
\mathrm{Var}(X_n) = \frac{n^3}{45} - \frac{3n^2}{40} - \frac{17n}{72} +\frac{35}{24}- \frac{1003}{360n} +\frac{157}{60n^2} - \frac{1}{n^3} \sim \frac{n^3}{45}.
\end{equation}

\subsection{Size Bias Transform for the Number of Crossings}

Let $X_n^s$ having the size bias distribution of the number of crossings $X_n$ defined in \eqref{crossing_number}. By Corollary \ref{corollary_sizebiasBernoulli}, we have that $X_n^s$ is given by
$$X_n^s = \sum_{j \in \binom{n}{4}} Y_j^{(I)},$$
where $I$ is a random index choose with $\mathbb{P}(I = i) = 1 /  \binom{n}{ 4}$ independent of  $\{Y_j \}$ and the distribution of $(Y_j^i)$ is given by $Y_i^i=1 $ and $(Y_j^i)_{j\neq i} \overset{(d)}{=}(Y_j)_{j \neq i} |\{Y_i = 1\}$.
With the previous construction, we have that  $Y_j^{(I)} = Y_j$ if $Y_I \perp Y_j$, which follows if  $I \neq j$, because disjoints trees are independent, so 
$$|X_n - X_n^s | = |\hat{X}_n - \hat{X}_n^s|$$
with
\begin{equation}\label{Xhat}
\hat{X}_n = \sum_{j \in \mathcal{C_I}} Y_j, \qquad \hat{X}_n = \sum_{j \in \mathcal{C_I}} Y_j^{(I)},
\end{equation}
where $\mathcal{C}_I$ is the set of crossings that are not disjoint with  $I$.
The number of crossings which are independent of a crossing $j$ is  $\binom{n-4}{4}$, which is the number of crossings that we can make with $n-4$ vertices. Thus, we have that for any crossing, 
\begin{align*}
|\mathcal{C}_I| &= \binom{n}{4} - \binom{n-4}{4} \\
&= \frac{n(n-1)(n-2)(n-3)}{24} - \frac{(n-4)(n-5)(n-6)(n-7)}{24} \\
&= \frac{2}{3}n^3 -7 n^2 +\frac{79}{3} n -35 := c_n.
\end{align*}
Moreover,  $j \in \mathcal{C}_i \Longleftrightarrow i \in \mathcal{C}_j$. Thus that define a symmetric (but not transitive) relation: $i \overset{\mathcal{C}}{\sim} j$ if $i \in \mathcal{C}_j$ (or $j \in \mathcal{C}_i$).

Finally, we explicitly construct the size bias transform $X^s$. Given a tree $T$, choose a crossing $I= (a,b,c,d)$ uniformly at random from $\binom{n}{4}$, independent of $T$.  We construct the tree $T^s$ as follows:

\begin{itemize}
\item If $I$ is a crossing of $T$, that is $Y_I = 1$, set $T = T^s$.
\item If $I$ is not a crossing of $T$, we consider the unique path $P_{ac}$ between $a$ and $c$ and we form a cycle putting the edge $ac$. Then we erase one of the edges incident to the vertices $a$ or $c$ in $P_{ac}$ with equal probability. We do the same process for the vertices $b$ and $d$. We set $T^s$ as the resulting tree.
\end{itemize}
With this construction, we have that $T^s$ is a tree with a crossing in the index $I$. $X^s$ is then the number of crossings considering $T^s$, instead of $T$.

\subsection{Bounding the Conditional Variance}

Finally,  our goal is to find a upper bound for the variance of the random variable given by the conditional expectation,  $\mathbb{E}[X_n^s-X_n |X_n]$.  First, we notice that 

\begin{align*}
\mathbb{E}[X_n^s - X_n | X_n] &=  \sum_{i \in \binom{n}{4} } \mathbb{E}[X_n^s - X_n | X_n, I = i] \mathbb{P}(I = i) \\
&= \frac{1}{\binom{n}{4}} \sum_{i \in \binom{n}{4}} (X_n^{(i)} - X_n),
\end{align*}
where $X_n^{(i)}$ denote $X_n^s$ conditioned to have a crossing in the index $i$. This gives that 
$$\mathrm{Var}(\mathbb{E}[X_n^s- X_n|X_n]) = \frac{1}{\binom{n}{4}^2} \sum_{i,j} \mathrm{Cov}(X_n^{(i)}-X_n, X_n^{(j)}- X_n). $$

To bound such covariances we will use the following lemma (see Lemma 2.6 in \cite{paguyo2021convergence} or Lemma 5.1 in \cite{he2022central}). We give a proof for he convenience of the reader.

\begin{lem}\label{cov}
Let $X$ and $Y$ be random variables with $|X| \leq C_1$ and $|Y| \leq C_2$. Let $A$ be some event such that conditional on $A$, $X$ and $Y$ are uncorrelated. Then 
$$|\mathrm{Cov}(X,Y)| \leq 4C_1C_2 \mathbb{P}(A^c).$$ 
\end{lem}
\begin{proof}
Let $\mu_X$ and $\mu_Y$ be the mean of $X$ and $Y$ respectively, then $|X-\mu_X| \leq 2C_1$ and $|Y-\mu_Y| \leq 2 C_2$. Now, if $A$ is a event such that $\mathbb{P}(A)>0$ and conditional on $A$, $X$ and $Y$ are uncorrelated, we obtain 
\begin{align*}
|\mathrm{Cov}(X,Y)| &= |\mathbb{E}[(X-\mu_X)(Y-\mu_Y)]| \\
&= |\mathbb{E}[(X-\mu_X)(Y-\mu_Y)\mathbb{I}_{A}] + \mathbb{E}[(X-\mu_X)(Y-\mu_Y)\mathbb{I}_{A^c}|]  \\
&=  |\mathbb{E}(X-\mu_X)(Y-\mu_Y)\mathbb{I}_{A^c}|\\
&\leq \mathbb{E}[|X-\mu_X||Y-\mu_Y|\mathbb{I}_{A^c}]\\ 
&\leq 4C_1C_2\mathbb{E}\mathbb{I}_{A^c}\\ 
&=4C_1C_2 \mathbb{P}(A^c).
\end{align*}
\end{proof}

Now, notice that $|X_n^{(i)} - X_n| \leq 4(n-3)$, because in order to obtain $X_n^{(i)}$ we want to have an specific crossing, and each edge introduces at most $n-3$ crossings. Now, we can identify two kinds of terms in the summation of the covariance: when the indices satisfy $|i \cap j| \neq0$ or when they satisfy $|i \cap j| =0$.

\begin{itemize}
\item \textbf{Case $|i \cap j| \neq 0$:} In the case we can have intersection of one, two, three or four vertices, we have 
$$\binom{n}{4} \left[ \binom{n-4}{3} + \binom{n-4}{2}+ \binom{n-4}{1} + 1 \right] \sim n^7$$
terms. Additionally, we have that for any index $i$,
$$\mathrm{Var}(X_n^{(i)}-X_n) \leq \mathbb{E}(X_n^{(i)}-X_n)^2 \leq 16(n-3)^2,$$
it follows that
$$|\mathrm{Cov}(X_n^{(i)}-X_n, X_n^{(j)}-X_n) \leq \sqrt{\mathrm{Var}(X_n^{(i)}-X_n)\mathrm{Var}(X_n^{(j)}-X_n)} \leq 16(n-3)^2.$$
Thus, the contribution of the terms $|i \cap j| \neq \emptyset$ is bounded by $O(n^7 n^2) = O(n^9)$.
\item \textbf{Case $|i \cap j | = 0$:} In this case we have $\binom{n}{4} \binom{n-4}{4} \sim n^8$ terms in the summation of the covariance. Let $A$ be the event in which the neighbours of the index are disjoints, that is 
$$A = \{ C_i \cap C_j  = \emptyset\}.$$
We have that conditional on the event $A$, $X_n^{(i)}-X_n$ and $X_n^{(j)} - X_n$ are independent. Indeed, 
$$X_n^{(i)}- X_n = \sum_{k \in C_i}(Y_k^{(i)}-Y_k),\quad \text{ and } \quad X_n^{(j)}- X_n = \sum_{l \in C_j}(Y_l^{(j)}-Y_l),$$
so each random variable depends only of the crossings $C_i$ and $C_j$ respectively, thus they are independent. 
Now, we have that if $k \in C_i \cap C_j$ then $i,j \in C_k$, this follows using the fact that $k \in C_i  \Leftrightarrow i \in C_k$, therefore in the set $A^c$ there are the crossings such that have vertices of $i$ and $j$. With this, we have 
$$\mathbb{P}(A^{c}) = \mathbb{P}(\{ k :  i,j \in C_k\}) \leq \mathbb{P}\left( \bigcup_{u \in i ,  v \in j  } \{u \sim v\} \right) \leq \sum_{u\in i, v \in j} \mathbb{P}(u \sim v) = \frac{32}{n}.$$

Using the Lemma \ref{cov}, 
$$|\mathrm{Cov}(X_n^{(i)}- X_n, X_n^{(j)}- X_n)| \leq 4(4(n-3))^2 \frac{32}{n} = \frac{2048(n-3)^2}{n},$$
thus, the contribution of this terms are bounded by $O(n^8 n) = O(n^9)$.
\end{itemize}
As a result, we obtain that 
$$\mathrm{Var}(\mathbb{E}[X_n^s- X_n|X_n]) = \frac{1}{\binom{n}{4}^2} \sum_{i,j} \mathrm{Cov}(X_n^{(i)}-X_n, X_n^{(j)}- X_n) \leq c n $$
for some constant $c$ independent of $n$. Thus, we proved the following

\begin{lem}
Let $X_n$ be the number crossing of a random labelled tree in convex position with $n$ vertices, and let $X_n^s$ have the size bias distribution with respect to $X_n$. Then
\begin{equation}\label{variance_bound}
\mathrm{Var}(\mathbb{E}[X_n^s- X_n|X_n]) \leq c n,
\end{equation}
where $c$ is a constant independent of $n$.
\end{lem}
\begin{remark}
    The constant $c$ in the above lemma could be taken to be  $2112$.
\end{remark}

\section{Rate of Convergence}

\subsection{Kolmogorov Distance}

Using the previous result, we obtain the hypothesis of Theorem \ref{teoKolmo}. Therefore, we find the convergence rates for the asymptotic normality of the number of crossings in a random labelled tree. 

\begin{teo}
\label{Th3}Let $X_n$ be the number of crossings of a random labelled tree in convex position with $n$ vertices. Let $\mu_n $ and $\sigma_n^2$  be the mean and the variance of $X_n$. Then, with $W_n = (X_n - \mu_n)/ \sigma_n$, 
\begin{equation}\label{rateKolmo}
d_{Kol}(W_n, Z) \leq \frac{C}{\sqrt{n}},
\end{equation}	
where $Z$ is a standard Gaussian random variable and $C$ is a constant independent of $n$.
\end{teo}

\begin{proof}
By construction we have that $|X_n^s - X_n | \leq 4(n-3)$. Also by \eqref{valoresperadocruces} and \eqref{varianzacruces}, 
$$\mu_n \sim n^2/6, \qquad \sigma_n^2 \sim n^3/45.$$
By Lemma \ref{rateKolmo}, 
$$\psi = \sqrt{\mathrm{Var}(\mathbb{E}(X_n^s-X_n |X_n))} \leq  \sqrt{cn},$$
then using Theorem \ref{teoKolmo}, 
\begin{align*}
d_{Kol}(W_n, Z) &\leq   \frac{6 \mu A^2}{\sigma^3} +\frac{2 \mu \Psi}{\sigma^2}\\ 
&\leq \frac{16n^2 (n-3)^2 }{\left(\sigma^2\right)^{3/2}} + \frac{\frac{n^2}{3}  \sqrt{cn}}{\sigma^2}\\
&\leq \frac{16n^4}{\left(\sigma^2\right)^{3/2}} + \frac{\sqrt{c} n^{5/2}}{3 \sigma^2} \\
&\leq \frac{C}{\sqrt{n}}.
\end{align*} 
\end{proof}

\subsection{Other Distances}
In the case that a bounded size bias coupling exists, it is possible to obtain bounds for a distances defined as supremums over functions classes, in which a Kolmogorov distance is a particular case. Following the Section 5.4 of \cite{chen2011normal}, we need the following condition over the class of functions.

\begin{con}\label{condition}$\mathcal{H}$ is a class of real valued measurable functions on $\mathbb{R}$ such that
\begin{enumerate}
\item The functions $h \in \mathcal{H}$ are uniformly bounded in absolute value by 1.
\item For any $c,d \in \mathbb{R}$ and $h(x) \in \mathcal{H}$, the function $h(cx+d) \in \mathcal{H}$.
\item For any $\epsilon >0$ and $h \in \mathcal{H}$, the functions $h_\epsilon^{+}, h_{\epsilon}^{-}$ are also in $\mathcal{H}$, and 
$$\mathbb{E} [h_\epsilon^{+}(Z)- h_{\epsilon}^{-}(Z)]  \leq a\epsilon,$$
for some constant $a$ that depends only of the class $\mathcal{H}$, where
$$h_\epsilon^{+}(x) = \sup_{|y| \leq \epsilon} h(x+y), \qquad h_{\epsilon}^{-}(x) =\inf_{|y| \leq \epsilon} h(x+y).$$
\end{enumerate}
\end{con}
Given a class $\mathcal{H}$ and random variables $X$ and $Y$,  let
$$\parallel \mathcal{L}(X) -   \mathcal{L}(Y) \parallel_{\mathcal{H}} = \sup_{h \in \mathcal{H}} | \mathbb{E}h(X)- \mathbb{E}h(Y)|.$$
A direct application of Theorem 5.8 from \cite{chen2011normal} in our case give the following result.

\begin{teo}
\label{Th4}
Let $X_n$ be the crossing number of a random labelled tree in convex position with $n$ vertices. Let $\mu_n $ and $\sigma_n^2$  the mean and the variance of $X_n$. Then for some class $\mathcal{H}$ that satisfy Condition \ref{condition} for some constant $a$, with $W_n = (X_n - \mu_n)/ \sigma_n$, 
\begin{equation}
\parallel \mathcal{L}(W_n) -   \mathcal{L}(Z) \parallel_{\mathcal{H}} \leq \frac{C_a}{\sqrt{n}},
\end{equation}	
where $Z$ is a standard Gaussian random variable and $C_a$ is a constant independent of $n$.
\end{teo}

As pointed out in \cite{chen2011normal}, naturally, the constant in Theorem \ref{Th3} is better that the one in Theorem \ref{Th4}.

\subsection*{Acknowledgement}

We would like to thank Professor Goldstein for pointing out the paper \cite{paguyo2021convergence} and for various communications during the preparation of this paper. OA would like to thank Clemens Huemer for initial discussions on the topic that led to work in this problem.

Santiago Arenas-Velilla was supported by a scholarship from CONACYT.

Octavio Arizmendi was supported by CONACYT Grant CB-2017-2018-A1-S-9764.
\\
{\begin{minipage}[l]{0.3\textwidth} \includegraphics[trim=10cm 6cm 10cm 5cm,clip,scale=0.15]{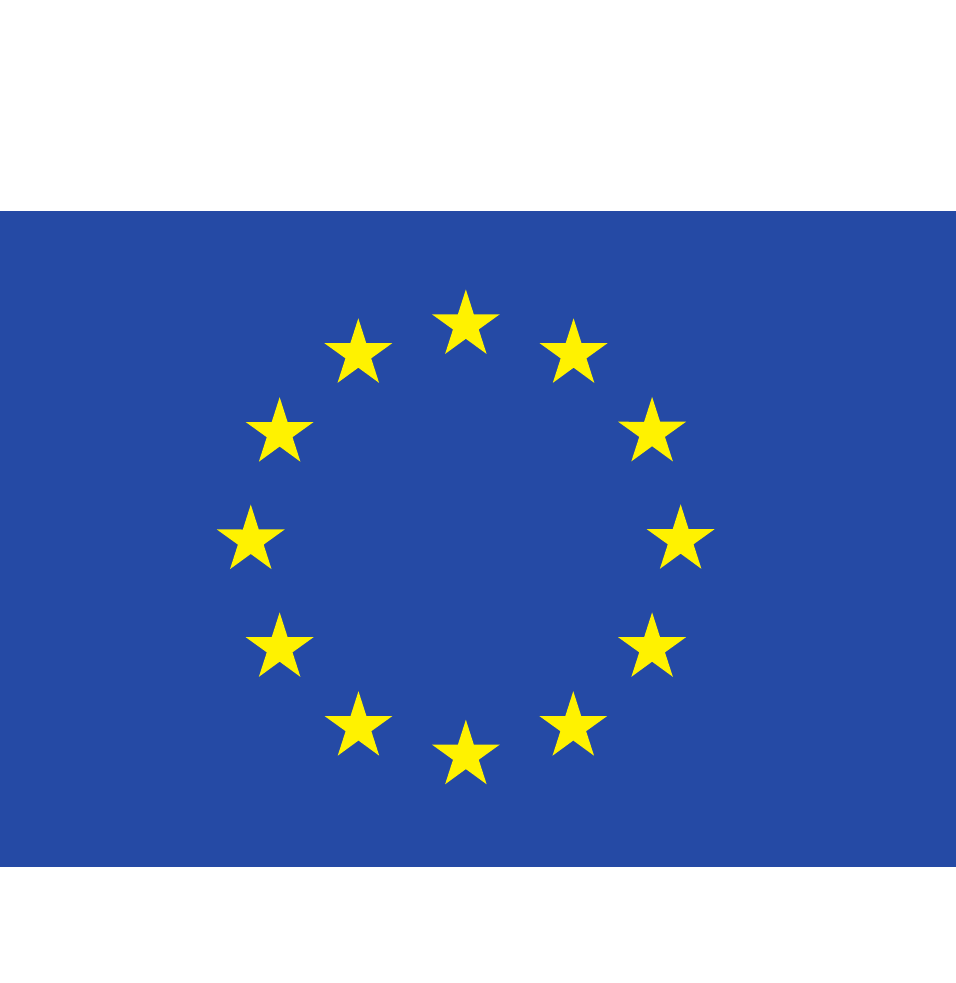} \end{minipage}
 \hspace{-3.2cm} \begin{minipage}[l][1cm]{0.82\textwidth}
 	  This project has received funding from the European Union's Horizon 2020 research and innovation programme under the Marie Sk\l{}odowska-Curie grant agreement No 734922.
		 	\end{minipage}}

\bibliographystyle{plain}
\bibliography{references}
\end{document}